\documentclass[12pt]{article}
\usepackage{amsmath,amssymb,color,amscd}
\usepackage{xspace}

\def\1{\underline{1}}

\def\Z{{\mathbb Z}}

\def\C{{\mathbb C}}

\def\O{{\mathcal O}}

\newtheorem{theorem}{Theorem}

\newtheorem{proposition}{Proposition}

\newenvironment{definition}
{\smallskip\noindent{\bf Definition\/}:}{\smallskip\par}

\newenvironment{remark}
{\smallskip\noindent{\bf Remark\/}.}{\smallskip\par}
\newenvironment{remarks}
{\smallskip\noindent{\bf Remarks\/}.}{\smallskip\par}
\newenvironment{proof}
{\noindent{\bf Proof\/}.}{{ $\square$}\smallskip\par}
\newenvironment{Proof}
{\noindent{\bf Proof\/}}{{$\square$}\smallskip\par}

\title{On simple $\Z_2$-invariant and corner function germs
}

\author{S.M.~Gusein-Zade 
\and A.-M.Ya.~Rauch \thanks{
The research was supported by the grant 16-11-10018 of the Russian Science Foundation.
Address: Moscow State University,
Faculty of Mechanics and Mathematics, GSP--1, Moscow, 119991, Russia.
E-mail: sabir\symbol{'100}mccme.ru, anakuta\symbol{'100}gmail.com}}
\date{}

\begin{document}

\maketitle

\begin{abstract}
V.I.~Arnold has classified simple (i.~e.\ having no modules for the classification) singularities
(function germs), and also simple boundary singularities (function germs invariant with respect to the action
$\sigma(x_1; y_1, \ldots, y_n)=(-x_1; y_1, \ldots, y_n)$ of the group $\Z_2$. In particular, it was shown that
a function germ (respectively a boundary singularity germ) is simple if and only if the intersection form
(respectively the restriction of the intersection form to the subspace to anti-invariant cycles) of a germ in $3+4s$
variables stable equivalent to the one under consideration is negative definite and if and only if the (equivariant)
monodromy group on the corresponding subspace is finite. We formulate and prove analogues of these statements
for function germs invariant with respect to an arbitrary action of the group $\Z_2$ and also for corner
singularities.
\end{abstract}

\section{Introduction}\label{sec:Introduction}
A celebrated result of V.I.~Arnold states that simple (that is having no moduli for the classification
(with respect to the right equivalence); see the precise definition in \cite[Part~II, Section~15]{AGV1})
germs of holomorphic functions up to the stable equivalence are given by the list $A_k$,
$D_k$, $E_6$, $E_7$, $E_8$: \cite{Arnold-1972}, see also \cite{AGV1}. Moreover, for a germ in $3\mod 4$ variables
the homology group (with integer coefficients) of the Milnor fibre together with the intersection form on it
is isomorphic to the
corresponding integer lattice with the scalar product (multiplied by $-1$ according to the traditions of the
Singularity Theory) and the monodromy group of this germ is isomorphic to the corresponding Weyl group.
Besides that a germ (in $3\mod 4$ variables) is simple if and only if the intersection form is negative definite
and if and only if its monodromy group is finite.

According to one of the interpretations, a {\em boundary singularity} is a germ of a (holomorphic) function on the
space $(\C^{n+1}, 0)$ invariant with respect to the action of the cyclic group $\Z_2$ of order two given by the
formula $\sigma(x_1; y_1, \ldots, y_n)=(-x_1; y_1, \ldots, y_n)$ ($\sigma$ is the non-zero element of the group
$\Z_2$). As this was shown in \cite{Arnold-1978}, simple boundary singularities up to the stable equivalence are
given by the list $A_k$, $D_k$, $E_6$, $E_7$, $E_8$, $B_k$, $C_k$, $F_4$, for a boundary singularity in $3\mod 4$
variables the subgroup of {\em anti-invariant} with respect to $\sigma$ elements of the homology group of the Milnor
fibre together with the intersection form on it is isomorphic to the corresponding integer lattice with the scalar
product multiplied by $-1$ and the monodromy group on it is isomorphic to the corresponding Weyl group.
Moreover, a germ in $3\mod 4$ variables is simple if and only if the intersection form on the subgroup of
anti-invariant cycles is negative definite and if and only if the corresponding monodromy group is finite.

Let the cyclic group $\Z_2$ of order 2 act on the germ of a complex affine space. Without loss of generality
one may assume that the action is linear and, moreover, it is defined by the representation $T_{m,n}$ of the group
$\Z_2$ on the space $(\C^{m+n}, 0)$ given by the formula
\begin{equation}\label{eq:action}
\sigma(x_1, \ldots, x_m; y_1, \ldots, y_n)=(-x_1, \ldots, -x_m; y_1, \ldots, y_n)\,,
\end{equation}
where $(x_1, \ldots, x_m)=\bar{x}\in\C^m$, $(y_1, \ldots, y_n)=\bar{y}\in\C^n$.
Germs invariant with respect to the described action of the group $\Z_2$ are considered up to an equivariant
coordinate changes: $f_1\sim f_2$ ($f_i:(\C^n, 0)\to(\C, 0)$, $f_i\circ\sigma=f_i$) if there exist a local
holomorphism (a holomorphic isomorphism) $h:(\C^n, 0)\to(\C^n, 0)$ such that $\sigma\circ h= h\circ\sigma$
and $f_2=f_1\circ h$. A stabilization of a $\Z_2$-invariant with respect to the representation $T_{m,n}$ germ
$f:(\C^{m+n}, 0)\to(\C, 0)$ (usually, i.~e.\ for $m\le 1$) is the invariant with respect to the representation
$T_{m,n+\ell}$ ($\ell\ge 0$) germ $\widehat{f}:(\C^{m+n+\ell}, 0)\to(\C, 0)$ given by the formula
$$
\widehat{f}(x_1, \ldots, x_m; y_1, \ldots, y_{n+\ell})=
f(x_1, \ldots, x_m; y_1, \ldots, y_n)+y_{n+1}^2+\ldots+y_{n+\ell}^2\,.
$$
In order to distinguish from another type of stabilization considered below, we shall call it {\em a stabilization
of the first type}, {\em I-stabilization} for short.
Two germs are (equivariantly) I-stable equivalent if some their I-stabilizations
are (equivariantly) equivalent. I-stabilization of a germ is equivariantly simple if and only if the germ
itself is equivariantly simple. This is a consequence of the equivariant version of the Morse Lemma with parameters
(which follows, for example, from the description of an invariant (mini)versal deformation of an invariant
germ \cite{Slodowy}). The results of Arnold from \cite{Arnold-1972} and \cite{Arnold-1978} describe, in particular,
all equivariantly simple germs $\Z_2$-invariant with respect the representation $T_{m,n}$ for $m=0$ and $1$
and state that in these cases a $\Z_2$-invariant germ is equivariantly simple if and only if the intersection form
on the subspace of the homology classes $a$ of the Milnor fibre of an I-stabilization of the germ depending on
$3\mod 4$ variables satisfying the condition $\sigma_*(a)=(-1)^ma$ is negative definite and if and only if the
corresponding monodromy group is finite. For short we shall call {\em $(-1)^{\bullet}$-invariant} the cycles from
the homology group of the Milnor fibre satisfying the condition $\sigma_*(a)=(-1)^ma$ .
Here we give the classification of the equivariantly simple germs invariant with respect to $\Z_2$-actions
(\ref{eq:action}) with other $m$ (up to the stable equivalence they are also given by the list
$A_k$, $D_k$, $E_6$, $E_7$, $E_8$, $B_k$, $C_k$, $F_4$) and show that the formulated criterion for simplicity
holds for these cases as well.

In the paper~\cite{Siersma} D.~Sirsma considered the problem of classification of the so-called corner singularities
wich can be interpreted as the problem of classification of function germs on the space $(\C^{m+n}, 0)$
invariant with respect to the action of the group $\Z_2^m$ given by the equation
\begin{equation}
\sigma_i(x_1, \ldots, x_m; y_1, \ldots, y_n)=(x_1, \ldots, x_{i-1}, -x_i, x_{i+1} x_m; y_1, \ldots, y_n)\,,
\end{equation}
where $\sigma_1$, \dots, $\sigma_m$ are the generators of the group $\Z_2^m$.
The paper~\cite{Siersma} contains a classification of the germs up to the codimension (more precisely, the
``bundle codimension'') four, i.~e.\ such that either they, or germs from the same $\{\mu={\rm const}\}$--families
can be in an irremovable way met in four-parameter families of functions. The paper does not contain a description
of simple germs apparently because it appears to be in some sense degenerate: up to the stable equivalence
and renumbering the generators of the group $\Z_2^m$ it coincides with the list of simple boundary singularities:
see Theorem~\ref{theo:class_corner}.
However, from the point of view of an analogue of the criterion of simplicity in terms of the negative definiteness
of the intersection form and of the finiteness of the monodromy group this classification leads to a non-trivial
(apriori not obvious) result. This is connected with the fact that the subspace of cycles on which the intersection
form is considered is, in general, strictly smaller than for the corresponding boundary singularities. The fact
that it is negative definite for simple singularities is of course a direct consequence of this property for
boundary singularities. However the fact that there are no other singularities possessing this property
(i.~e. that the confining singularities do not possess this property) does not follow from the corresponding
statement for boundary singularities and needs its own verification (at least for some of them).

\section{Equivariantly simple $\Z_2$-invariant germs}\label{sec:Z_2-simple}
\begin{theorem}\label{theo:class}
 Equivariantly simple function germs invariant with respect to the action $T_{m,n}$ of the group $\Z_2$ on
 the space $\C^{m+n}$ up to the equivariant I-stable equivalence are given by the following list:
 \begin{enumerate}
  \item[1)] $A_k$: $x_1^2+\ldots+x_m^2+y_1^{k+1}$, $k\ge 1$;
  \item[2)] $D_k$: $x_1^2+\ldots+x_m^2+y_1^2y_2+y_2^{k-1}$, $k\ge 4$;
  \item[3)] $E_6$: $x_1^2+\ldots+x_m^2+y_1^3+y_2^4$;
  \item[4)] $E_7$: $x_1^2+\ldots+x_m^2+y_1^3+y_1y_2^3$;
  \item[5)] $E_8$: $x_1^2+\ldots+x_m^2+y_1^3+y_2^5$;
  \item[6)] $B_k$: $x_1^{2k}+x_2^2+\ldots+x_m^2$, $k\ge 2$;
  \item[7)] $C_k$: $x_1^2y_1+x_2^2+\ldots+x_m^2+y_1^k$, $k\ge 2$;
  \item[8)] $F_4$: $x_1^4+x_2^2+\ldots+x_m^2+y_1^3$.
 \end{enumerate}
\end{theorem}

\begin{definition}
A {\em stabilization of the second type} ({\em II-stabilization} for short) of a $\Z_2$-invariant with respect
to the representation $T_{m,n}$ germ
$f:(\C^{m+n}, 0)\to(\C, 0)$ is the invariant with respect to the representation $T_{m+k,n}$ ($k\ge 0$)
germ $\widehat{f}:(\C^{m+k+n}, 0)\to(\C, 0)$ defined by the formula
$$
\widehat{f}(x_1, \ldots, x_{m+k}; y_1, \ldots, y_{n})=
f(x_1, \ldots, x_m; y_1, \ldots, y_n)+x_{m+1}^2+\ldots+x_{m+k}^2\,.
$$
By a {\em stabilization} of a germ we shall call a II-stabilization of its I-stabilization.
We shall say that two germs are {\em (equivariantly) II-stable equivalent}, if some their
II-stabilizations are (equivariantly) equivalent. We shall say that two germs are {\em (equivariantly) stable
equivalent} if some their stabilizations are (equivariantly) equivalent.
\end{definition}

Theorem~\ref{theo:class} can be formulated in the following form: a $\Z_2$-invariant germ is simple if and only if
it is II-stable equivalent to a simple boundary germ (i.~e.\ to a germ with $m=1$).

\begin{Proof} of Theorem~\ref{theo:class} to a big extent is already contained in~\cite{Arnold-1978}.
For a $\Z_2$-invariant (with respect to the representation $T_{m,n}$) germ $f:(\C^{m+n}, 0)\to(\C, 0)$
with a critical point at the origin, let us denote by $m_1$ and $n_1$ the coranks of the restrictions of the germ
$f$ to the subspaces $\C^m$ and $\C^n$ respectively, i.~e.\ the coranks of their second differentials.
For stably equivariantly equivalent germs these coranks coincide. The germ $f$ is equivariantly equivalent to a
stabilization of a germ $\check f:(\C^{m_1+n_1}, 0)\to(\C, 0)$ $\Z_2$-invariant with respect to the representation
$T_{m_1,n_1}$. Moreover, the germ $f$ is simple if and only if the germ $\check f$ is simple. Simple germs with
$m_1\le 1$ were classified in \cite[\S8]{Arnold-1978}. If $m_1\ge 1$, a germ stably equivalent to $\check f$ is
adjacent to a germ stably equivalent to an (invariant with respect to the representation $T_{2,0}$) germ of the form
$P_4(x_1,x_2)$, where $P_4$ is a non-degenerate homogeneous polynomial of degree 4. Such a germ obviously is not
simple. (It is equivalent to a germ $x_1^4+x_2^4+ a x_1^2x_2^2$ with $a^2\ne 4$, where $a$ is a modulus which
determines the cross ration of the lines constituting the zero level set.) Together with \cite[\S8]{Arnold-1978}
this proves the statement. A germ stably equivalent to $P_4(x_1,x_2)$ we shall call a germ of type $M_5$. The name
is chosen more or less randomly taking in mind the Arnold's names $K_{4,2}$ and $L_6$ for confining boundary
singularities: see Proposition~\ref{prop:bound} below. The index 5 indicates ``the Milnor number'', i.~e.\ the
dimension of the space of $\Z_2$-invariant cycles in the Milnor fibre of the germ $P_4(x_1,x_2)$.
\end{Proof}

The proof of Theorem~\ref{theo:class} together with \cite[page~268]{AGV1} and \cite[Corollary~1]{Arnold-1978}
implies the following statement.

\begin{proposition}\label{prop:bound}
 Any non-simple germ is adjacent to a germ stably equivalent to one from the following
 {\em confining families}:
 \begin{enumerate}
  \item[1)] $P_8$: $y_1^3+y_2^3+y_3^3+ay_1y_2y_3$, $a^3+27\ne 0$;
  \item[2)] $X_9$: $y_1^4+y_2^4+ay_1^2y_2^2$, $a^2\ne 4$;
  \item[3)] $J_{10}$: $y_1^3+y_2^6+ay_1^2y_2^2$, $4a^3+27\ne 0$;
  \item[4)] $F_{1,0}$: $x_1^6+y_1^3+ax_1^2y_1^2$, $4a^3+27\ne 0$;
  \item[5)] $K_{4,2}$: $x_1^4+y_1^4+ax_1^2y_1^2$, $a^2\ne 4$;
  \item[6)] $L_6$: $x_1^2y_1+ax_1^2y_2+ y_1^3+y_2^3$, $a^3\ne 1$;
  \item[7)] $M_5$: $x_1^4+x_2^4+ax_1^2x_2^2$, $a^2\ne 4$.
 \end{enumerate}
\end{proposition}

\begin{remark}
 We use the notations of Arnold. ``More modern'' notations for the first three singularities are
 $\widetilde{E}_6$, $\widetilde{E}_7$ and $\widetilde{E}_8$. 
\end{remark}

The following statement uses the notion of the equivariant monodromy group which, according to our knowledge,
has no fixed definition in the literature. For boundary singularities a precise definition was given
in~\cite{Arnold-1978}. Here we shall give a general definition for an arbitrary germ invariant with respect
to an action of a finite group which will be used below.

Let $f:(\C^k,0)\to(\C,0)$ be a germ of a holomorphic function invariant with respect to an action of a finite
group $G$ on $(\C^k,0)$ and having an isolated critical point at the origin. Let $\widetilde{f}:U\to\C$ be
a generic $G$-invariant perturbation of the germ $f$ общего вида (defined in a small neighbourhood $U$ of the origin.
The genericity condition consists of the requirement that, for $G$ invariant functions close to $\widetilde{f}$,
the sum of the Milnor numbers of the critical point lying on one level set is the same as for the function
$\widetilde{f}$. In other words, all critical points of the function $\widetilde{f}$ are non-splittable and each
level set of the function $\widetilde{f}$ contains not more than one $G$-orbit of critical points.
For germs invariant with respect to an action of the group $\Z_2$ and for corner singularities, non-splittableness
is equivalent to non-degeneracy (i.e.\ being Morse). The monodromy transformations corresponding to going around
the critical values of the function $\widetilde{f}$ can be assumed to be $G$-equivariant (as self-maps of the
Milnor fibre). Therefore the corresponding monodromy operators (the Picard–Lefschetz operators:
the actions of the monodromy transformations on
the $(k-1)$st homology group of the Milnor fibre) commute with the representation of the group $G$ on the
homology group (and therefore preserve the subspaces corresponding to irreducible representations of the group $G$).
The group of operators on the homology group of the Milnor fibre generated by the monodromy operators corresponding
to going around of all the critical values is called {\em the equivariant monodromy group} of the germ $f$.
It is not difficult to show that the notion is well-defined, i.e.\ does not depend on the choice of the perturbation
$\widetilde{f}$ (since the space of generic perturbations is connected). We shall consider the action of the
equivariant monodromy group on the subspace of the homology group corresponding to one particular irreducible
representation of the group $\Z_2$ or of $\Z_2^m$~--- the subspace of the $(-1)^{\bullet}$-invariant cycles.

\begin{theorem}\label{theo:topology}
A $\Z_2$-invariant (with respect to the action $T_{m,n}$) germ is simple if and only if the intersection form on
the subspace of $(-1)^{\bullet}$-invariant cycles in the homology group of the Milnor fibre of its I-stabilization
with the number of variables congruent to three modulo 4 is negative definite and if and only if
its equivariant monodromy group on this subspace is finite.
\end{theorem}

\begin{proof}
According to the description of the intersection form of a stabilization of a germ (\cite{Gabrielov}, see also
\cite[Part~I, Section~2, Theorem~14]{AGV2}) and taking into account the fact that the subspace of
$(-1)^{\bullet}$-invariant cycles for a stabilization of a germ (including the one of the second type) is isomorphic
to the subspace of $(-1)^{\bullet}$-invariant cycles of the germ itself (the latter follows, e.~g., from
\cite[Part~I, Section~2, Theorem~9]{AGV2}), it is sufficient to verify the statement for one representative of 
the stable equivalence class of a germ. The fact that all germs equivariantly simple with respect to a $\Z_2$-action
possess the indicated property follows from \cite{AGV1} and \cite{AGV2}. 
(In fact this directly follows from the fact that all these germs considered as ``usual'' (non-invariant)
are simple.) In order to prove that there are no other
germs possessing this property, it is sufficient to verify that the confining singularities from
Proposition~\ref{prop:bound} do not possess the property. For the first six this was verified by Arnold (and
formulated in \cite{AGV1} and \cite{AGV2}). Therefore it is sufficient to verify the statement for the (invariant
with respect to the action $T_{2,1}$) germ $x_1^4+x_2^4+y_1^2$. According to the description of Dynkin diagrams
of germs of functions in two variables (or rather of germs stable equivalent to germs of functions in two variables;
\cite{GZ}, see also \cite[Part~I, Section~4]{AGV2}) this germ has the Dynkin diagram given on Figure~\ref{fig:1}.
(The numbering of the vertices corresponds to the numbering of the basis cycles $\Delta_i$ in a distinguished basis.)
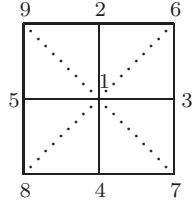
\begin{figure}[h]
$$
\unitlength=0.50mm
\begin{picture}(60.00,50.00)(0,10)
\thinlines
\put(10,50){\line(1,0){40}}
\put(10,30){\line(1,0){40}}
\put(10,10){\line(1,0){40}}
\put(10,10){\line(0,1){40}}
\put(30,10){\line(0,1){40}}
\put(50,10){\line(0,1){40}}
\multiput(12,12)(2,2){19}%
{\circle*{1}}
\multiput(48,12)(-2,2){19}%
{\circle*{1}}
\put(30,33){{\scriptsize$1$}}
\put(9,52){{\scriptsize$9$}}
\put(29,52){{\scriptsize$2$}}
\put(49,52){{\scriptsize$6$}}
\put(9,4){{\scriptsize$8$}}
\put(29,4){{\scriptsize$4$}}
\put(49,4){{\scriptsize$7$}}
\put(6,28){{\scriptsize$5$}}
\put(52,28){{\scriptsize$3$}}
\end{picture}
$$
\caption{Dynkin diagram of the function $x_1^4+x_2^4+y_1^2$.}
\label{fig:1}
\end{figure}
The group $\Z_2$ acts on the basis cycles by the reflection with respect to the center of the diagram
(i.~e.\ preserving the cycle $\Delta_1$ and mapping the cycles $\Delta_i$ and $\Delta_{i+2}$
into each other for $i=2,3,6,7$). Therefore the subspace of $(-1)^{\bullet}$-invariant (in this case~---
invariant) cycles is generated by the cycles $\delta_1=\Delta_1$, $\delta_2=\Delta_2+\Delta_4$,
$\delta_3=\Delta_3+\Delta_5$, $\delta_4=\Delta_6+\Delta_8$ and $\delta_5=\Delta_7+\Delta_9$.
(The dimension of the subspace of these cycles, as it was indicated above,
is equal to five.) It is not difficult to see that the restriction of the intersection form to this subspace
is not negative definite. Namely, the cycles $\nabla=2\Delta_1 +(\Delta_2+\Delta_4)+(\Delta_3+\Delta_5)$ and
$\nabla'=(\Delta_2+\Delta_4)+(\Delta_3+\Delta_5)+(\Delta_6+\Delta_8)+(\Delta_7+\Delta_9)$
$\Delta_1 +\frac{1}{2}(\Delta_2+\Delta_4)+\frac{1}{2}(\Delta_3+\Delta_5)$
lye in the kernel of this form. The equivariant monodromy group on the subspace of $(-1)^{\bullet}$-invariant 
(invariant in this case) cycles is generated by the Picard--Lefschetz operators $h_i$ associated to going around
the corresponding critical values. The operators $h_i$ are the reflections in the corresponding basis cycles
$\delta_i$ (with respect to the intersection form). The latter follows from the following reasoning (given for
simplicity for $h_2$). The definition of the operators $h_i$ implies that $h_2(a)=H_4H_2(a)$, where $H_i$ is
the reflection (of the whole homology group) in the basis cycles $\Delta_i$.
According to the classical Picard--Lefschetz formula we have
$$
h_2(a)=H_4H_2(a)=a+(a,\Delta_2)\Delta_2+(a,\Delta_4)\Delta_4+(a,\Delta_2)(\Delta_2,\Delta_4)\Delta_4\,.
$$
Moreover $(\Delta_2,\Delta_4)=0$ and, for an invariant cycle $a$ we have $(a,\Delta_2)=(a,\Delta_4)$.
This implies that
$$
h_2(a)=\frac{(a,\delta_2)}{2}\delta_2
$$
and therefore on the subspace of the invariant cycles the operator $h_2$ is the reflection in the cycle $\delta_i2$:
it is the identity on the orthogonal complement to the cycle $\delta_2$ and it maps $\delta_2$ to $-\delta_2$
(note that $(\delta_2,\delta_2)=-4$). The fact that this group is not finite
can be deduced from the general classification of finite groups generated by reflections and also follows from
the following computation. It is not difficult to see that $h_5h_4h_1(\delta_2+\delta_3)=\delta_2+\delta_3+\nabla$.
Since $\nabla$ lies in the kernel of the intersection form, $h_i(\nabla)=\nabla$. Therefore
$(h_5h_4h_1)^s(\delta_2+\delta_3)=\delta_2+\delta_3+s\nabla$ what implies that the group generated by the reflections
$h_i$ is infinite.
\end{proof}

\begin{remarks}
 {\bf 1.} According to~\cite{Wall} the subspace of $(-1)^{\bullet}$-invariant cycles in the homology group with
 complex coefficients of the Milnor fibre of a germ $f$ is isomorphic to the subspace of invariant elements in the
 local algebra $\O_{\C^{m+n},0}/\langle \frac{\partial\,f}{\partial x_i}, \frac{\partial\,f}{\partial y_j}\rangle$
 ($\O_{\C^{m+n},0}$ is the ring of germs of holomorphic functions on $\C^{m+n}$ at the origin).
 However this isomorphism is not canonical and an analogue of the intersection form on the local algebra
 is not well-defined.
 
 {\bf 2.} A very interesting problem is to understand whether analogues of the criterion of simplicity of a germ
 formulated in Theorem~\ref{theo:topology} holds in more general situations, say, for germs invariant with respect
 to actions of other finite groups. The main problem is the fact that up to now there is no conceptual proof
 of this criterion permitting to understand reasons why it holds. Its proofs in \cite{AGV1} and \cite{AGV2} (and
 here as well) consist of independent classifications of simple germs and of germs with negative definite
 intersection form (or its restriction) and comparison of these classifications (when it appears that they
 coincide). Because of that, for example for germs invariant with respect to actions of other finite groups,
 at this moment one cannot see an approach which permits to avoid a classification of simple germs which
 up to now does not exist. (For the cyclic group $\Z_3$ of order three it was obtained recently by the
 second author of this paper.)
\end{remarks}

\section{Simple corner singularities}\label{sec:corner}
Let us consider the action $S_{m,n}$ of the group $\Z_2^m$ on the space $\C^{m+n}$ given by the formula
\begin{equation}\label{eq:corner_action}
\sigma_i(x_1, \ldots, x_m; y_1, \ldots, y_n)=(x_1, \ldots, x_{i-1}, -x_i, x_{i+1} x_m; y_1, \ldots, y_n)\,,
\end{equation}
where $\sigma_1$, \dots, $\sigma_m$ are the generators of the group $\Z_2^m$, $(x_1, \ldots, x_m)=\bar{x}\in\C^m$,
$(y_1, \ldots, y_n)=\bar{y}\in\C^n$. The equivariant equivalence classes of germs of holomorphic functions on
$(\C^{m+n},0)$ invariant with respect to the representation $S_{m,n}$ are called {\em corner singularities}
(see \cite{Siersma}). We shall use the notions of the I-stable equivalence and of the stable equivalence in the
form they were defined in Section~\ref{sec:Z_2-simple}.

\begin{theorem}\label{theo:class_corner}
 Simple corner germs up to the equivariant I-stable equivalence and to renumbering of the generators of the group
 $\Z_2^m$ are given by the list 1)--8) from Theorem~\ref{theo:class}.
 Confining families for simple corner singularities (i.e.\ such that each non-simple germ is adjacent to a germ
 stable equivalent to one of them) are the germs given by the equations 1)--6) from Proposition~\ref{prop:bound} and
\begin{enumerate}
 \item[7)] $M_4$: $x_1^4+x_2^4+ax_1^2x_2^2$, $a^2\ne 4$.
\end{enumerate}
\end{theorem}

{\bf The proof} is obtained by a small modification of the proof of Theorem~\ref{theo:class}.

\begin{remark}
 The first part of Theorem~\ref{theo:class_corner} means that a corner singularity is simple if and only if it
 is stable equivalent to a simple boundary germ.
 The list of the confining families looks completely coinciding with the list from Proposition~\ref{prop:bound}
 However, the last of the singularities from this list ($M_4$) coincides with the last of the singularities from
 the list of Proposition~\ref{prop:bound} ($M_5$) only formally: as a ``usual'' (non-invariant) singularity. Here
 $M_4$ is considered as a $\Z_2^2$-invariant function, whence $M_5$ is considered as a $\Z_2$-invariant one.
 The sense of the index $4$ in the notation of the singularity will be explained later (see the proof of
 Theorem~\ref{theo:corner_topology}).
\end{remark}

We shall call an element of the homology group of the Milnor fibre of a corner singularity
{\em $(-1)^{\bullet}$-invariant}, if it is antiinvariant with respect to the action of each generator
$\sigma_i$, $i=1, \ldots, m$.
This means that it belongs to the invariant with respect to the action of the group $\Z_2^m$ subspace of the homology
group of the Milnor fibre corresponding to the (one-dimensional) representation $\det S_{m,n}$ on the $(m+n)$th
exterior power of the space $\C^{m+n}$. 

\begin{theorem}\label{theo:corner_topology}
 A germ of a corner singularity is simple if and only if the intersection form on the subspace of
 $(-1)^{\bullet}$-invariant cycles in the homology group of the Milnor fibre of its I-stabilization with the
 number of variables congruent to three modulo 4 is negative definite and if and only if the equivariant monodromy
 group on this subspace is finite.
\end{theorem}

\begin{proof}
As in Theorem~\ref{theo:topology} it is sufficient to verify the statement for one representative of the
stable equivalence class of a germ. The fact that all corner singularities from the list of simple ones possess
the indicated property follows from the fact that they are stable equivalent to boundary ones. In order to prove
that other germs possessing this property do not exist, it is sufficient to verify that the confining corner
singularities do not possess this property. For the first six this follows from the fact that they are stable
equivalent to confining boundary singularities. Therefore it is sufficient to verify the statement for the germ
$x_1^4+x_2^4+ax_1^2x_2^2+y_1^2$ ($a^2\ne 4$) I-stable equivalent to the germ $M_4$. The Dynkin diagram of this germ
is given on Figure~\ref{fig:1}. The generators $\sigma_1$ and $\sigma_2$ of the group $\Z_2^2$ act on the basis
cycles by the reflections respectively with respect to the vertical and in the horizontal axis with the
multiplication by $-1$. Therefore the subspace of $(-1)^{\bullet}$-invariant cycles is generated by the cycles
$\delta_1=\Delta_1$,
$\delta_2=\Delta_2+\Delta_4$, $\delta_3=\Delta_3+\Delta_5$, $\delta_4=\Delta_6+\Delta_7+\Delta_8+\Delta_9$.
(The dimension of this subspace is equal to four, what is indicated in the notation of the singularity. The
restriction of the intersection form to this subspace has the kernel generated by
$\nabla=2\delta_1 +\delta_2+\delta_3$ and $\nabla'=\delta_2+\delta_3+\delta_4$. Just as in the case of
$\Z_2$-invariant germs (Theorem~\ref{theo:topology}), the equivariant monodromy group on the
subspace of $(-1)^{\bullet}$-invariant cycles is generated by the reflections $h_i$ in the listed basis cycles
$\delta_i$. One can see that $h_4h_1(\delta_2+\delta_3)=\delta_2+\delta_3+\nabla$. Therefore
$(h_4h_1)^s(\delta_2+\delta_3)=\delta_2+\delta_3+s\nabla$, what implies that the group generated by the reflections
$h_i$ is infinite.
\end{proof}

\bigskip
\rightline{Translated by S.M.Gusein-Zade.}
\end{document}